\documentclass[12pt,final]{amsart}

\usepackage{graphicx} 
\usepackage[margin=1.0in]{geometry}
\usepackage{amsmath, amssymb, amsfonts}

\usepackage{verbatim}
\usepackage{amssymb}
\usepackage{amscd}

\newtheorem{lemma}{Lemma}[section]

\newtheorem{theorem}[lemma]{Theorem}

\theoremstyle{definition}
\newtheorem{defn}[lemma]{Definition}
\theoremstyle{remark}

\newtheorem{example}[lemma]{Example}

\author{Tatyana Barron}
\address{Department of Mathematics, University of Western Ontario, London Ontario N6A 5B7, Canada}
\email{tatyana.barron@uwo.ca}
\date{\today}

\author{Kai Boisvert}
\address{Department of Mathematics, University of Western Ontario, London Ontario N6A 5B7, Canada}
\email{kboisve2@uwo.ca}

\author{Noah Vale}
\address{Department of Mathematics, University of Western Ontario, London Ontario N6A 5B7, Canada}
\email{nvale3@uwo.ca}

\thanks{}

\title{On vector-valued multisymplectic forms}
\begin{document}
\sloppy

\maketitle

\noindent {\bf Abstract.} We obtain a standard local presentation for a vector-valued multisymplectic form on a smooth manifold, generalizing the known proof for polysymplectic forms.    
We show that vector-valued multisymplectic forms on a finite-dimensional real vector space form a non-unital operad. We prove an entropy inequality for partial compositions.

\


\


\section{Introduction}

A symplectic form on a smooth manifold $M$ is a closed nondegenerate $2$-form. In differential geometry and applications (especially, to mathematical physics) this definition has been generalized in two ways: 

for a positive integer $k$, a $k$-plectic form on $M$ is a closed nondegenerate $(k+1)$-form; 

for a positive integer $m$, a polysymplectic form is  an ${\mathbb{R}}^m$-valued closed nondegenerate $2$-form on $M$. 

The canonical symplectic form on the cotangent bundle $M=T^*X$ has an appropriate generalization to a multisymplectic form on the multicotangent bundle. This is well known. For polysymplectic manifolds,  C. Blacker \cite{blacker} used a similar approach to obtain a standard local representation of a  polysymplectic form. In Section \ref{sec:darboux}, we extend this result to ${\mathbb{R}}^m$-valued $k$-plectic forms: Theorem \ref{th:canf}.    

A lot of results in symplectic geometry or Poisson geometry are viewed through the classical-quantum perspective. One of such famous results is the proof of the formality conjecture by M. Kontsevich. Operads play an important role in understanding this breakthrough \cite{kontsev}. 

The information theoretic entropy (its various versions) is commonly used      
with algebraic structures such as graphs. See the introduction of \cite{lein} for a discussion about entropy on operads. 
In Section \ref{sec:operad}, we show that on a real finite-dimensional vector space $V$,  the  ${\mathbb{R}}^m$-valued $k$-plectic forms with a fixed value of $k$ \  form a non-unital operad and we discuss entropy. This is Theorem 
\ref{th:operad}.  
We provide examples. 

Geometry of multisymplectic and polysymplectic manifolds has a rich history. In multisymplectic geometry, there has been a lot of effort to formulate and understand the appropriate analogues of the fundamental concepts and facts from symplectic geometry. Our definition of a polysymplectic manifold is the same as in \cite{blacker}, \cite{blackerq}.  For example, in \cite{hirota}, the authors distinguish between polysymplectic and $V$-symplectic manifolds (where $V$ is a finite-dimensional real vector space).  
In \cite{awane}, \cite{rroy}, a $k$-symplectic structure consists of $k$ closed $2$-forms on an $n(k+1)$-dimensional manifold, with a transversality condition. In \cite{hirota}, the authors extend the considerations to the  manifolds with vector bundle-valued multisymplectic forms. 
The analogues of polysymplectic Darboux theorem have been obtained in \cite{blacker}, \cite{forgomes}, \cite{rroy}. Attempts to formulate a  multisymplectic Darboux theorem lead to the obvious complications (see e.g. \cite{cantr}, \cite{rw}). Our theorem in Section \ref{sec:darboux}
is an extension of Blacker's result in \cite{blacker}. As symplectic geometry has its origins in classical mechanics, the development of multisymplectic and  polysymplectic formalism is linked with field theory. Applications to mathematical physics have many interesting facets.  It is well known that a hyperk\"ahler manifold is multisymplectic ($3$-plectic, with its fundamental $4$-form), and on the other hand it is polysymplectic (as a manifold with $3$ symplectic forms, interpreted as an ${\mathbb{R}}^3$-valued symplectic form).   
There are concrete physical systems represented by multisymplectic or polysymplectic manifolds: see the original paper \cite{nambu}, or for example \cite{rupakc}. For applications to hydrodynamics see e.g. \cite{mspera}. Starting with \cite{nambu}, there has been an extensive discussion around the classical-quantum correspondence principle. To give some references, we mention 
\cite{barron}, \cite{barser}, \cite{barsha}, \cite{blackerq}, \cite{rogers}, \cite{ss}. There are applications to string theory (see e.g. 
\cite{baezhr}) and to calibrated and special geometries ($G_2$, $Spin(7)$): \cite{dkarig}, \cite{mads1}, \cite{mads2}. 
Operads naturally appear in conformal field theory (see the table on the second page of \cite{gk}) and in geometric quantization \cite{kontsev}. 
As mentioned in both these papers, it is a useful perspective to work with the algebras over the associative operad, or over the Lie operad, etc. The author of \cite{mathi} points out that his construction is an operad.    
In symplectic geometry, the geometric quantization perspective is essentially that a symplectic form  $\omega$ 
on a smooth manifold $M$ is a beginning of a discussion of quantization on $M$. Similarly, if a smooth manifold is equipped with a multisymplectic or polysymplectic structure, that already means that there are intrinsic quantum aspects. The operad ${\mathcal{P}}$ from our Theorem \ref{th:operad}  has no obvious relationship to the multisymplectic $L_{\infty}$ algebras. We will address the algebras over this operad, ${\mathcal{P}}$,  in future work.   

{\bf Acknowledgements.} We are thankful to Anibal Medina-Mardones and Bruno Vallette for discussions. This paper originates from work started in 2025, when  K.B. was partially supported by a 2025 USRI, funded by the University of Western Ontario and A. Medina-Mardones' NSERC grant RES000678, and N.V. was partially supported by a 2025 NSERC USRA.
 
\section{Towards an analogue of the Darboux theorem} 
\label{sec:darboux}

Let $n$ be an integer such that $n\ge 2$. 
Let $k$ be a positive integer such that $1\le k\le n-1$. 
Let $m\in{\mathbb{N}}$. 

\subsection{The linear case}
Let $V$ be an $n$-dimensional real vector space. 
\begin{defn}
\label{def:linform}
Let $n,k,m,V$ be as above. 
Let $\omega$ be   a multilinear map 
$$
\omega: {\bigwedge}^{k+1} V\to {\mathbb{R}}^m. 
$$
We say $\omega$ is {\bf   nondegenerate} if the following condition holds: 

if $v\in V$ is such that  $\omega(v,u_1,...,u_k)=0\in {\mathbb{R}}^m$ \  for all $u_1,...,u_k\in V$, then 
$v=0\in V$. 

\noindent We say $\omega$ is a ${\mathbb{R}}^m$-{\bf valued $k$-plectic form} on $V$
if it  is nondegenerate.  

For each $i\in \{ 1,...,m\}$, let $\omega_i$ be the $i$-th component of $\omega$, i.e. 
$$
\omega_i:{\bigwedge}^{k+1} V\to {\mathbb{R}}
$$
is defined by $\omega_i=p_i\circ \omega$, where $p_i:{\mathbb{R}}^m\to {\mathbb{R}}$ is the projection 
$x=(x_1,...,x_m)\mapsto x_i$. 

Denote by $\bar \omega_i$ the map 
$$
\bar \omega_i: V\to ({\bigwedge}^{k} V)^*
$$
induced by $\omega_i$: for $v\in V$, and  $u_1,...,u_k\in V$ 
$$
\bar \omega_i(v)(u_1,...,u_k)=\omega_i(v, u_1,...,u_k).
$$
\end{defn}
\begin{lemma}
\label{lem:nondeg}
Let $n,k,m\in {\mathbb{N}}$, $n\ge 2$, \  $1\le k\le n-1$. 
 Let $V$ be an $n$-dimensional real vector space. 
Let $\omega$ be   a multilinear map 
$\omega: {\bigwedge}^{k+1} V\to {\mathbb{R}}^m$. 
For each $i\in \{ 1,...,m\}$, let $\omega_i$ be the $i$-th component of $\omega$ and let 
$\bar \omega_i$ denote the map 
$\bar \omega_i: V\to ({\bigwedge}^{k} V)^*$
induced by $\omega_i$, as in the Definition \ref{def:linform}. Then $\omega$ is nondegenerate if and only if 
\begin{equation}
\label{eq:kerint}
\bigcap\limits_{i=1}^m \ker (\bar \omega_i)=\{ 0\} .
\end{equation}
\end{lemma}
{\bf Proof.} $\Rightarrow$  \ Suppose $\omega$ is nondegenerate. Assume $v\in V$ is such that $v\in \ker\bar \omega_i$ for all
$i\in \{ 1,...,m\}$.   Then, for any  choice of $u_1,...,u_k\in V$, and for each 
$i\in \{ 1,...,m\}$
$$
\omega _i(v,u_1,...,u_k)=0.
$$
Since $\omega$ is nondegenerate, it follows that $v=0$. 

$\Leftarrow$ \ 
Suppose (\ref{eq:kerint}) holds. We would like to show that $\omega$ is nondegenerate. Let us prove the contrapositive. Assume there exists  $v\in V$ such that $v\ne 0$ and $\omega(v,u_1,...,u_k)=0\in {\mathbb{R}}^m$ for all $u_1,...,u_k\in V$. Then $\omega_i(v,u_1,...,u_k)=0\in {\mathbb{R}}$ for all $u_1,...,u_k\in V$ and all $i\in \{ 1,...,m\}$. 
Hence $v\in \bigcap\limits_{i=1}^m \ker (\bar \omega_i)$. We showed: if $\omega$ is not nondegenerate, then 
$\bigcap\limits_{i=1}^m \ker (\bar \omega_i)\ne \{ 0\}$. 
$\Box$

When $m=1$, Definition \ref{def:linform} yields a $k$-plectic (multisymplectic) form. When $k=1$, it is a definition 
of a polysymplectic form on $V$.  
Here are two examples with $k=1$ (polysymplectic forms on $V$). 
\begin{example}
Let $n$ be an even integer, let $V$ be an $n$-dimensional real vector space. Let $m\in {\mathbb{N}}$. 
Let $\omega_1$,..., $\omega_m$ be linear symplectic forms on $V$.   Then 
$$
\omega: {\bigwedge}^{2} V\to {\mathbb{R}}^m 
$$
defined by
$$
\omega (u,v)=\begin{pmatrix} 
\omega_1(u,v) \\ .... \\ \omega_m(u,v)
\end{pmatrix}
$$
is an ${\mathbb{R}}^m$-valued $1$-plectic form on $V$. In this example, each $\omega_i$ is nondegenerate. \end{example}
\begin{example}
\label{exmp:crossp}
Let $m=n=3$ and $k=1$. Thus, $V={\mathbb{R}}^3$.  
Define
$$
\omega: {\bigwedge}^{2} V\to {\mathbb{R}}^3
$$
by setting $\omega (u,v)$ to be the cross product of the vectors $u$ and $v$ in ${\mathbb{R}}^3$, i.e.  
$$
\omega (u,v)=\begin{pmatrix} 
\omega_1(u,v) \\ \omega_2(u,v) \\ \omega_3(u,v) 
\end{pmatrix}
$$
where 
$$
\omega_1(u,v)=
u_2v_3-u_3v_2; \ \omega_2(u,v)= 
-u_1v_3+u_3v_1; \ 
 \omega_3(u,v) =
u_1v_2-u_2v_1 .
$$
The forms $\omega_1$, $\omega_2$, $\omega_3$ are not nondegenerate. The form $\omega$ is nondegenerate. 
In the terminology of our Definition \ref{def:linform}, it is an ${\mathbb{R}}^3$-valued $1$-plectic form on $V$.  \end{example}

\subsection{Local canonical form} 
\begin{defn}
Let $E_m\to Y$ be the trivial rank $m$ real vector bundle on an $n$-dimensional manifold $Y$, i.e. 
$E_m=Y\times {\mathbb{R}}^m$. 
An ${\mathbb{R}}^m$-{\bf valued $k$-plectic form} on $Y$ is 
 a smooth section $F$ of $\bigwedge^{k+1}T^*Y\otimes E_m$ such that  $F$ is nondegenerate at every $y\in Y$ (i.e. $F_y$ is an ${\mathbb{R}}^m$-valued $k$-plectic form on the vector space $T_yY$) and $dF=0$. 
\end{defn}
We note: for a $(k+1)$-form $\alpha$ and an ${\mathbb{R}}^m$-valued function $f$,  

\noindent $d(\alpha\otimes f)=d\alpha\otimes f+(-1)^{k+1}\alpha \wedge df$. 
 
 Examples of multisymplectic manifolds ($m=1$, arbitrary $k$) can be found in \cite{cantr}, \cite{rw}, 
  examples of polysymplectic manifolds ($k=1$, arbitrary $m$) can be found in \cite{blacker}.  
It is well known that a naive attempt to generalize 
the Darboux theorem to the multisymplectic case does not work. In the linear case: the general linear group of an even-dimensional real vector space $V$ acts transitively on the space of skew-symmetric nondegenerate 2-forms on $V$, but the description of orbits of $GL(V)$ on the space of $k$-plectic forms on a real vector space $V$ is more complicated (see e.g. the discussion and references in \cite{cantr}).  To rephrase, for a given $k>1$, and a given real vector space $V$, there may be several linearly inequivalent normal forms of  a $k$-plectic form on $V$. 

On a manifold, it is not true that a multisymplectic form is locally of a constant linear type (for a counterexample, see Lemma 4.15 and Example 4.16 in \cite{rw}). 

Theorem 1.1 \cite{blacker} (also see in \cite{blacker} the proof of Theorem 3.6) gives a local form for an aribitrary polysymplectic form on a manifold $M$, analogous to the local expression of a symplectic form on a manifold in terms of the 
canonical symplectic form on its cotangent bundle. 
While, arguably, this is not a full strength "Darboux-type result", it does give a local presentation of an arbitrary 
polysymplectic form. 
 The theorem that we prove below extends this result to vector-valued multisymplectic forms.

Let $M$ be a smooth $n$-dimensional manifold. Assume $n\ge 2$. 
Let $k$ be a positive integer such that $1\le k\le n-1$. 
Let $m\in{\mathbb{N}}$. Let $E_m\to M$ be the trivial vector bundle $M\times {\mathbb{R}}^m\to M$. 
Denote by 
$$\pi=\pi_{k,m}:{\bigwedge} ^kT^*M\otimes E_m\to M
$$ the projection. Denote by $\Theta=\Theta_{k,m}$ the canonical ${\mathbb{R}}^m$-valued $k$-form  on the manifold 
$X=X_{k,m}= \bigwedge^kT^*M\otimes E_m$. 
For $m=1$ and $k=1$, $\Theta$ is the Liouville $1$-form on $T^*M$, and the $2$-form $d\Theta$ is the canonical symplectic form 
on $T^*M$. In the more general case, when 
$m=1$, and $k$ is an arbitrary positive integer such that $k<n$,  \  $\Theta$ is the $k$-form from \cite{cantr} section 6. 
See also \cite{baezhr} sec. 2. Here, we are generalizing to  the case when $m$ is an arbitrary positive integer.   
The form $\Theta$ is defined by: 

for $(x,\eta)\in X$, where $x\in M$ and $\eta$ in the fiber of the bundle $\bigwedge^kT^*M\otimes E_m$ over $x$,   and $v_1,...,v_k\in T_{(x,\eta)}X$
\begin{equation}
\label{eq:tetadef}
\Theta(v_1,...,v_k)=\eta(\pi_*(v_1),...,\pi_*(v_k)).
\end{equation}
Let $\Omega=d\Theta$. 

It will be useful to write $\Theta$ and $\Omega$ in a local coordinate chart $(U,\phi)$ on $M$, with coordinates 
$(q_1,...,q_n)$. For a multiindex $I=(i_1,...,i_k)$ with $1\le i_1< ...< i_k\le n$ we write 
$$
dq_I=dq_{i_1}\wedge ....\wedge dq_{i_k} 
$$
and we denote by ${\mathcal{J}}$ the set of multiindices $I$ above. 
The forms $dq_I$ form a local basis in the fibers of $\bigwedge^kT^*M$.  Let $e_1$,...,$e_m$ be a frame for $E_m$.
Then, the local coordinates on $X$ are $\{ q_1,...,q_n; p_{I,j} \ |  \ 1\le j\le m; \ I\in {\mathcal{J}}\}$, 
and $\eta=\sum\limits_{j=1}^m\sum\limits_I p_{I,j} \ dq_I\otimes e_j$. 
\begin{lemma} 
\label{lem:oneform}
In the notations above, 

(a) \ $\Theta=\sum\limits_{i=1}^m\sum\limits_I p_{I,i} \ dq_I\otimes e_i$

(b) \ $\Omega=\sum\limits_{i=1}^m\sum\limits_I d(p_{I,i}) \wedge dq_I\otimes e_i$

(c) \ $\Omega$ is an ${\mathbb{R}}^m$-valued $k$-plectic form on $X$.
\end{lemma} 
Remark: For $m=1$, part (c)  is stated in \cite{cantr}, \cite{baezhr}, and proved in \cite{rw}. 

{\bf Proof of Lemma \ref{lem:oneform}.} 
Recall that $\pi_*:T({\bigwedge} ^kT^*M\otimes E_m)\to TM$, and 
for $v\in T_{(x,\eta)}X$ which is a linear combination of $\dfrac{\partial}{\partial q_i}$ and 
$\dfrac{\partial}{\partial (p_{I,i})}$,  \ 
$\pi_*v$ is the part of this sum which involves the terms with $\dfrac{\partial}{\partial q_i}$. Evaluating the left hand side 
of (\ref{eq:tetadef}) on arbitrary $v_1$,..., $v_k$, we conclude that 
$\Theta=\sum\limits_{i=1}^m\sum\limits_I f_{I,i} \ dq_I\otimes e_i$ (for some coefficients 
$f_{I,i}$), and furthermore 
$\Theta=\sum\limits_{i=1}^m\sum\limits_I p_{I,i} \ dq_I\otimes e_i$. 

This proves (a). Then, $\Omega=d\Theta$, and (b) follows. 

Since $\Omega$ is an exact $(k+1)$-form on $X$, it is closed. Let's show that $\Omega$ is nondegenerate. 
Suppose 
$v\in T_{(x,\eta)}X$ is such that $v\lrcorner \Omega =0$. Writing  
$v\in T_{(x,\eta)}X$ as a linear combination of $\dfrac{\partial}{\partial q_i}$ and $\dfrac{\partial}{\partial (p_{I,i})}$, 
and using (b), we deduce that every coefficient in this linear combination is zero, and therefore $\Omega$ is nondegenerate. 
$\Box$ 

\begin{theorem} 
\label{th:canf}
Let all the notations, including $M$, $n$, $m$, $k$, X, $\Theta$, $\Omega$ be as defined above. Suppose 
$\omega=\omega_{k,m}$ is an ${\mathbb{R}}^m$-valued 
$k$-plectic form on $M$, i.e. 
$$
\omega=\begin{pmatrix}\omega_1 \\ ... \\ \omega_m\end{pmatrix}
$$
where, for each $i\in \{ 1,..., m\}$, $\omega_i$ is a closed $(k+1)$-form on $M$, and 
$\omega$ is nondegenerate. Then for every $x\in M$ there exists a neigborhood $U=U_x$ of $x$ in $M$ and a local smooth embedding $f=f_x:U\to X$ such that $\omega\Bigr |_U=f^*\Omega$.   
\end{theorem}

{\bf{Proof}}. Let $x_0\in M$. 
Since each $\omega_i$ is closed, by the Poincar\'e Lemma there exists a neighborhood $U$ of $x_0$, diffeomorphic to an $n$-dimensional ball in ${\mathbb{R}}^n$ centered at $x_0$, in a local chart at $x_0$, such that on $U$, 
$\omega_i=d\alpha_i$ for $k$-forms $\alpha_1$,..., $\alpha_m$.  Define the embedding $f:U\to X$ by 
$f(x)=(x,\alpha)$, where 
$$
\alpha=\alpha_x=\begin{pmatrix} \alpha_1(x) \\ ... \\ \alpha_m(x)\end{pmatrix} 
$$

Let's show: over $U$, $f^*\Theta=\alpha$. Let $x\in U$. Let $\xi_1,...,\xi_k\in T_x(U)$. 
We get: 
$$
(f^*\Theta)(\xi_1,...,\xi_k)=\Theta(f_*(\xi_1),...,f_*(\xi_k)).
$$
From the definitions of $\Theta$ and $f$
$$
\Theta(f_*(\xi_1),...,f_*(\xi_k))=\alpha(\pi_*f_*(\xi_1),...,\pi_*f_*(\xi_k)). 
$$
Because for each $i\in \{1,...,k\}$,  \ $\pi_*f_* (\xi_i)=\xi_i$, we conclude 
$(f^*\Theta)(\xi_1,...,\xi_k)=\alpha(\xi_1,...,\xi_k)$. We showed that  on $U$, \  $f^*\Theta=\alpha$. Since $d\alpha=\omega$ and $d\Theta=\Omega$, it follows that $\omega=f^*\Omega$. 
$\Box$

The following two examples are an illustration of Theorem \ref{th:canf}. 
\begin{example}
Let $M={\mathbb{R}}^3$, $m=3$, $k=1$, and $\omega$ is the polysymplectic (${\mathbb{R}}^3$-valued) form defined in Example  \ref{exmp:crossp}:
$$
\omega=\begin{pmatrix} \omega_1\\ \omega_2\\ \omega_3\end{pmatrix}
$$
where 
$$
\omega_1 =dq_2\wedge dq_3; \ \omega_2 =dq_3\wedge dq_1; \ \omega_3 =dq_1\wedge dq_2.
$$
Choose the $1$-forms $\alpha_j$, $j\in \{1,2,3\}$, such that $\omega _j=d\alpha_j$ as follows: 
$$
\alpha_1 =q_2 \ dq_3; \ \alpha_2 =q_3 \ dq_1; \ \alpha_3 =q_1 \ dq_2.
$$
We have: 
$$
T^*M\simeq {\mathbb{R}}^3\times {\mathbb{R}}^3; \ X=T^*M\otimes {\mathbb{R}}^3.
$$
Let $\{ e_1, e_2, e_3\}$ be the standard basis in ${\mathbb{R}}^3$. Over $x=(q_1,q_2,q_3)\in M$, 
we write $\eta\in T_q^*M\otimes {\mathbb{R}}^3$ as follows:
$$
\eta=(\sum_{i=1}^3p_{1,i} \ dq_1\otimes e_i)+(\sum_{i=1}^3p_{2,i} \ dq_2\otimes e_i)+
(\sum_{i=1}^3p_{3,i} \ dq_3\otimes e_i).
$$
We have:
$$
\Omega=\begin{pmatrix}
dp_{1,1}\wedge dq_1+dp_{2,1}\wedge dq_2+dp_{3,1}\wedge dq_3 \\ 
dp_{1,2}\wedge dq_1+dp_{2,2}\wedge dq_2+dp_{3,2}\wedge dq_3 \\ 
dp_{1,3}\wedge dq_1+dp_{2,3}\wedge dq_2+dp_{3,3}\wedge dq_3 
\end{pmatrix} .
$$
We embed $M$ into $X$ via 
$$
f:M\to X
$$
$$
x\mapsto (x,\alpha).
$$
Therefore, for the points of $X$ that are in $f(M)$, we have 
$$
p_{1,2}=q_3; \ p_{2,3}=q_1; \ p_{3,1}=q_2
$$
and all other $p_{i,j}$ are zero. So, we observe explicitly, that $f^*\Omega=\omega$ (which is the statement of the theorem).   
\end{example}
\begin{example}
Let $M={\mathbb{R}}^6$, $m=1$, $k=2$, and $\omega$ is the multisymplectic ($2$-plectic) form 
from 
Example 4.16 and  Lemma 4.15 of \cite{rw}: 
$$
\omega=dx_1\wedge dx_3\wedge dx_5  - dx_1\wedge dx_4\wedge dx_6 -dx_2\wedge dx_3\wedge dx_6
+x_2dx_2\wedge dx_4\wedge dx_5.
 $$
Choose a $2$-form $\alpha$ such that $\omega=d\alpha$ as follows: 
$$
\alpha=x_1 \  dx_3\wedge dx_5  - x_1 \  dx_4\wedge dx_6 -x_2 \  dx_3\wedge dx_6
+\frac{1}{2}x_2^2 \  dx_4\wedge dx_5.
 $$
We have: 
$$
T^*M\simeq {\mathbb{R}}^6\times {\mathbb{R}}^6; \ X={\bigwedge}^2T^*M.
$$
The local coordinates on $X$ are $(q_1,q_2,q_3,q_4,q_5,q_6;p_{ij} \ | \ 1\le i<j\le 6\}$, 
where $\{ p_{ij}\}$ are the fiber coordinates with respect to the basis $\{ dq_i\wedge dq_j\}$. We have: 
$$
\Omega=\sum_{1\le i<j\le 6}dp_{ij}\wedge dq_i\wedge dq_j.
$$
We embed $M$ into $X$ via 
$$
f:M\to X
$$
$$
x\mapsto (x,\alpha).
$$
Therefore, for the points of $X$ that are in $f(M)$, we have 
$$
p_{35}=q_1; \ p_{46}=-q_1; \ p_{36}=-q_2; \ p_{45}=\frac{1}{2}q_2^2
$$
and all other $p_{ij}$ are zero. Therefore, we can see explicitly that $f^*\Omega=\omega$. 
\end{example}

\section{Operads and entropy}
\label{sec:operad}

\begin{defn}\label{def:dop1} [ Definition 1.1 \cite{kimura}, Definition 12.1.1 \cite{lein}, Section 1 \cite{may} ] An {\bf operad}  ${\mathcal{P}}$  is a collection of sets $P(n)$, 
$n\ge 1$, with an action of the symmetric group $\Sigma_n$ on $P_n$, a composition law 
 $$
\lambda: P(k)\times P(n_1)\times ...\times P(n_k)\to P(n_1+ ...+n_k)
$$
$$
(f;f_1,...,f_k)\mapsto \lambda (f;f_1,...,f_k)
$$
and a unit $e\in P(1)$, such that the following properties are satisfied: 

(i) the composition is equivariant with respect to the symmetric group actions: 

for $\sigma\in \Sigma_{k}$, $\sigma_1\in \Sigma_{n_1}$ \ , ...,  \ $\sigma_k\in \Sigma_{n_k}$, $f\in P(k)$, $f_1\in P(n_1)$ \ ,...,  \ $f_k\in P(n_k)$
$$
(f;f_1\sigma_1,...,f_k\sigma_k)=( \lambda (f;f_1,...,f_k))\sigma'
$$
where $\sigma'$ is the image of $(\sigma_1, ..., \sigma_k)$ under the obvious map  
$\Sigma_{n_1}\times ...\times \Sigma_{n_k}\to 
\Sigma_{n_1+...+n_k}$, and 
$$
(f\sigma ;f_1,...,f_k)=( \lambda (f;f_{\sigma ^{-1}(1)},...,f_{\sigma ^{-1}(k)}))\sigma''
$$
where $\sigma ''\in \Sigma_{n_1+...+n_k}$ acts like $\sigma$ on the set of $k$ blocks of the sizes 
$n_1$,...,$n_k$ (acting as the identity map within each of the blocks). 

(ii) the composition is associative:
$$
\lambda( \ \lambda(\xi;\phi_1,...,\phi_n) \ ; \ \psi_{11},...,\psi_{1k_1},  \  \psi_{21},...,\psi_{2k_2},..., \ \psi_{n1},...,\psi_{nk_n})=
$$
$$
\lambda( \ \xi; \ \lambda(\phi_1; \ \psi_{11},...,\psi_{1k_1}),  \ 
\lambda (\phi_2;  \  \psi_{21},...,\psi_{2k_2}), ..., \ \lambda(\phi_n; \ \psi_{n1},...,\psi_{nk_n}))
$$

(iii) for each $f\in P_k$ we have $\lambda (e;f)=f$ and $\lambda(f;\underbrace{e,...,e}_{k})=f$.

\end{defn}
\begin{defn} \label{def:dop2}
[\cite{loday} sec. 5.3.4] An {\bf operad}  ${\mathcal{P}}$  is a collection of sets $P(n)$, 
$n\ge 1$, with an action of the symmetric group $\Sigma_n$ on $P_n$, and partial  compositions 
 $$
\circ_i: P(n)\times P(m)\to P(n+m-1)
$$
for $i\in \{ 1,...,n\}$, 
and a unit $e\in P(1)$, such that the following properties are satisfied: 

(i) the partial compositions are compatible with the symmetric group actions, as governed by the condition (i) of Definition \ref{def:dop1}, 

(ii) for $f\in P(l)$, $g\in P(m)$, $h\in P(n)$
$$
(f \circ_k h)\circ_i g=(f \circ_i g)\circ_{k-1+m} h
$$
for $1\le i<k\le l$; and 
$$
f \circ_i (g\circ_j h)=(f \circ_i g)\circ_{i-1+j} h
$$
for $1\le i\le  l$, \ $1\le j\le m$. 

(iii) for each $f\in P_k$ we have $e\circ_1 f=f$ and $f\circ_i e=f$.

\end{defn}
Equivalence between the Definitions \ref{def:dop1} and \ref{def:dop2} is established by the proof of Prop. 5.3.4 in \cite{loday}. Either definition can be modified to a definition of a {\it non-unital operad} (Remark 1.1 \cite{kimura}): a collection of sets $P(n)$, $n\ge 2$, with a composition and a symmetric group action, satisfying (i) and (ii).     

\begin{theorem} \label{th:operad}
Let $n$ be an integer such that $n\ge 2$. 
Let $k$ be a positive integer such that $1\le k\le n-1$. 
Let $V$ be an $n$-dimensional real vector space. 

For each integer $m\ge 2$, let $P(m)$ be the set of ${\mathbb{R}}^m$-valued $k$-plectic forms on $V$. 
Let $\circ_i$ be defined by: 
$$
\circ_i: P(p)\times P(q)\to P(p+q-1),
$$
for 
$$
\beta=\begin{pmatrix} \beta_1 \\ ...\\ \beta_p \end{pmatrix}; \ 
\alpha=\begin{pmatrix} \alpha_1 \\ ...\\ \alpha_q \end{pmatrix}
$$
\begin{equation}
\label{eq:circi}
\beta\circ_i \alpha=\begin{pmatrix} \beta_1 \\ ...\\ \beta_{i-1} \\ \alpha_1 \\ ...\\ \alpha_q \\ \beta_{i+1} \\ ... \\ 
\beta_p \end{pmatrix}.
\end{equation}
Also, for $\alpha\in P(q)$ and for a choice of $v_1,...,v_{k+1}\in V$, such that for each $j\in \{1,...,q\}$ 
$\alpha_j(v_1,...,v_{k+1})\ne 0$,  define 
\begin{equation}
\label{eq:defentr}
E(\alpha)=-\sum_{j=1}^q \frac{(\alpha_j(v_1,...,v_{k+1}))^2}{\sum\limits_{i=1}^q(\alpha_i(v_1,...,v_{k+1}))^2}
\ln \frac{(\alpha_j(v_1,...,v_{k+1}))^2}{\sum\limits_{i=1}^q(\alpha_i(v_1,...,v_{k+1}))^2} 
\end{equation}
(we suppress the choice of the $k+1$ vectors in the notation for $E$).

Then 

(a) the sets $P(m)$ and the compositions $\circ_i$ define a non-unital operad; 

(b) for $\alpha\in P(q)$, and for a choice of $v_1,...,v_{k+1}\in V$, such that 
for each $j\in \{1,...,q\}$ 
$\alpha_j(v_1,...,v_{k+1})\ne 0$, we have 
$$
E(\alpha\circ_i \alpha)\le 2E(\alpha)+\ln 2;
$$

(c) for $\beta\in P(p)$, $\alpha\in P(q)$,  and for a choice of $v_1,...,v_{k+1}\in V$, such that 
for each $j\in \{1,...,p\}$  \ 
$B_j:=(\beta_j(v_1,...,v_{k+1}))^2\ne 0$, 
and for each $j\in \{1,...,q\}$ 
$A_j=(\alpha_j(v_1,...,v_{k+1}))^2\ne 0$, 
and such that for $i\in \{ 1,...,p\}$, $B_i=A:=\sum\limits_{j=1}^qA_j$, 
we have 
$$
E(\beta\circ_i \alpha)= E(\beta)+\frac{A}{B}E(\alpha)
$$
where $B$ denotes $\sum\limits_{j=1}^pB_j$; 

(d) for a $(k+1)$-form $\gamma^{(q)}=\begin{pmatrix} \gamma  \\ ... \\ \gamma \end{pmatrix}\in P(q)$ 
and a choice of $v_1,...,v_{k+1}\in V$, such that $\gamma(v_1,...,v_{k+1})=c\ne 0$, 
$$
E(\gamma ^{(q)})=\ln q. 
$$
\end{theorem}
{\bf Remarks}.
\begin{itemize}
\item In the assumptions of the Theorem, $V$ and $k$ are fixed. The integer $m$ is allowed to vary and to take arbitrary values that are strictly greater than $1$.   
\item In (c), for each $i$, we can always normalize $\alpha$ to satisfy $A=B_i$. Indeed, define 
$$
\tilde{\alpha}=\sqrt{\frac{B_i}{A}} \ \alpha.
$$
Then $\tilde{\alpha}\in P(q)$ and 
$$
\sum\limits_{j=1}^q(\tilde{\alpha}_j(v_1,...,v_{k+1}))^2=\sum\limits_{j=1}^q\frac{B_i}{A}A_j=B_i.
$$

\item The expression (\ref{eq:defentr}) is the standard information-theoretic entropy. It has the maximum value $\ln q$. 
Thus, the entropy in (d) is the maximum entropy. We also note that (d) applies to the canonical 
$k$-plectic ${\mathbb{R}}^q$-valued form on $V$.       
\end{itemize}
{\bf Proof.} Proof of (a). We will  verify the Definition \ref{def:dop2}, modified to a nonunital operad. First, we observe that 
$\circ_i$ is well-defined, meaning that $\beta\circ_i\alpha$ defined by (\ref{eq:circi}) is indeed in $P(p+q-1)$. 
It is a ${\mathbb{R}}^{p+q-1}$-valued $(k+1)$-form on $V$.  Because 
$\alpha$ is nondegenerate, by Lemma \ref{lem:nondeg},  $\bigcap\limits_{i=1}^q \ker (\bar \alpha_i)=\{ 0\}$. 
Then 
$$
(\bigcap\limits_{i=1}^q \ker (\bar \alpha_i))\cap \ker (\bar \beta_1)\cap...\cap \ker (\bar \beta_{i-1})  
\cap \ker (\bar \beta_{i+1})\cap ... \cap \ker (\bar \beta_{p}) =\{ 0\}.
$$
Hence $\beta\circ_i\alpha$ is nondegenerate. 

Both equalities in part (ii) of Definition \ref{def:dop2} follow directly from (\ref{eq:circi}). It remains to verify part (i) of Definition \ref{def:dop2}.  Let $\beta\in P(p)$ and $\alpha\in P(q)$. Using (\ref{eq:circi}), we obtain the following. 
For $\sigma\in \Sigma_q$ 
$$
\beta\circ_i(\alpha\sigma)= ( \beta\circ_i \alpha)\sigma'
$$
where $\sigma'$ refers to the image of $\sigma$ under the obvious homomorphism $\Sigma_q\to \Sigma_{p+q-1}$. 
For $\sigma\in \Sigma_p$ 
$$
(\beta \sigma)\circ_i \alpha= ( \beta\circ_{\sigma^{-1}(i)}\alpha)\sigma''
$$
where $\sigma''$ is the element of $\Sigma_{p+q-1}$ that acts like $\sigma$ on the set of $p$ blocks (the $p-1$ blocks of size $1$ and $1$ block of size $q$), and acts as the identity map within the size $q$ block. 

Proof of (b). To simplify notations, we write a proof for $i=1$. The proof for an arbitrary $i$ is similar. Denote  
$$
A=\sum_{j=1}^q (\alpha_j(v_1,...,v_{k+1}))^2; 
A_1=(\alpha_1(v_1,...,v_{k+1}))^2.
$$
We have: 
$$
E(\alpha)=-\sum_{j=1}^q \frac{(\alpha_j(v_1,...,v_{k+1}))^2}{A}\ln \frac{(\alpha_j(v_1,...,v_{k+1}))^2}{A}
$$
$$
(\alpha\circ_1\alpha) (v_1,...,v_{k+1})=\begin{pmatrix}  \alpha_1 (v_1,...,v_{k+1})\\ ...\\ 
\alpha_q (v_1,...,v_{k+1})\\ 
\alpha_2 (v_1,...,v_{k+1})\\ ... \\ 
\alpha_q (v_1,...,v_{k+1})\end{pmatrix} 
$$
$$
E(\alpha\circ_1 \alpha)=
-\frac{1}{2A-A_1}
\Bigl (
A_1\ln \frac{A_1}{2A-A_1}+
2\sum_{j=2}^q (\alpha_j(v_1,...,v_{k+1}))^2\ln \frac{(\alpha_j(v_1,...,v_{k+1}))^2}{2A-A_1} \Bigr ) =
$$
$$
-\frac{1}{2A-A_1}
\Bigl (
2\sum_{j=1}^q (\alpha_j(v_1,...,v_{k+1}))^2\ln (\alpha_j(v_1,...,v_{k+1}))^2-
$$
$$
A_1\ln A_1-2A\ln (2A-A_1)+A_1\ln (2A-A_1)\Bigr ) =
$$
$$
-\frac{1}{2A-A_1}
\Bigl (
-2AE(\alpha)+2A\ln A
-A_1\ln A_1-2A\ln (2A-A_1)+A_1\ln (2A-A_1)\Bigr ) \le 
$$
$$
2E(\alpha)+\ln (2A-A_1)
-\frac{2A\ln A}{2A-A_1}+\frac{A_1\ln A_1}{2A-A_1} \le 
2E(\alpha)+\ln (2A)
-\frac{2A\ln A}{2A-A_1}+\frac{A_1\ln A}{2A-A_1}=
$$
$$
2E(\alpha)+\ln 2.
$$
Proof of (c). Using (\ref{eq:circi}), and taking into consideration that $A=B_i$, we get: 
    \begin{flalign*}
        E(\beta \circ_i \alpha) &= -\left(\sum_{j \neq i} \frac{B_j}{B} \ln \frac{B_j}{B}\right) - \left(\sum_{r=1}^{q} \frac{A_r}{B} \ln \frac{A_r}{B}\right) \\
        &= -\left(\sum_{j \neq i} \frac{B_j}{B} \ln \frac{B_j}{B}\right) - \left(\sum_{r=1}^{q} \left(\frac{B_i}{B} \ \frac{A_r}{A}\right) \left(\ln \frac{B_i}{B} + \ln \frac{A_r}{A}\right)\right) \\
        &= -\left(\sum_{j \neq i} \frac{B_j}{B} \ln \frac{B_j}{B}\right) - \left(\left(\frac{B_i}{B} \ln \frac{B_i}{B}\right) \frac{1}{A}\sum_{r=1}^{q} A_r\right) - \left(\frac{B_i}{B}\sum_{r=1}^{q}  \frac{A_r}{A} \ln \frac{A_r}{A}\right) \\
        &= E(\beta) + \frac{B_i}{B}E(\alpha).
    \end{flalign*}
Proof of (d). Each of the $q$ terms in the sum (\ref{eq:defentr}) equals \  $\dfrac{1}{q}\ln\dfrac{1}{q}$, and the conclusion follows.  
$\Box$
\begin{example} Consider, in the operad of Theorem \ref{th:operad} on $V={\mathbb{R}}^3$, with $k=1$, \  $\omega
=\begin{pmatrix} \omega_1  \\ \omega_2 \\ \omega_3\end{pmatrix} \in P(3)$ 
defined by the cross product of vectors (Example \ref{exmp:crossp}). Choose and fix $u,v\in V$ such that 
$c_i=\omega_i(u,v) \ne 0$ for each $i\in \{ 1,2,3\}$. We have: 
$$
\omega\circ_1\omega=\begin{pmatrix}  \omega_1 \\ \omega_2 \\ 
\omega_3 \\ 
\omega_2 \\ 
\omega_3 \end{pmatrix} \in P(5); \ 
(\omega\circ_1\omega)\circ_1\omega=\begin{pmatrix}  \omega_1 \\ \omega_2 \\ 
\omega_3 \\ 
\omega_2 \\ 
\omega_3 \\
\omega_2 \\ 
\omega_3
\end{pmatrix}\in P(7)
$$
and after applying $\circ_1\omega$ consecutively $j$ times, we get an element of $P(3+2j)$ whose value on $u,v$ is this  vector 
in ${\mathbb{R}}^{3+2j}$:
$$ 
   \begin{pmatrix}  \omega_1 (u,v)\\ \omega_2 (u,v) \\ 
\omega_3 (u,v)\\ 
\omega_2 (u,v)\\ 
\omega_3 (u,v)\\ 
... \\
\omega_2 (u,v)\\ 
\omega_3 (u,v)
\end{pmatrix},
$$
the value of entropy is
$$
E=-\frac{1}{c_1^2+(j+1)(c_2^2+c_3^2)}\Bigl (  c_1^2\ln \frac{c_1^2}{c_1^2+(j+1)(c_2^2+c_3^2)}
+(j+1)c_2^2\ln \frac{c_2^2}{c_1^2+(j+1)(c_2^2+c_3^2)}+
$$
\begin{equation}
\label{eq:entr3}
(j+1)c_3^2\ln \frac{c_3^2}{c_1^2+(j+1)(c_2^2+c_3^2)}\Bigr )
\end{equation}
and we also point out that the {\it {disorder}} $D=\dfrac{E}{E_{max}}$  \  (see \cite{davison}, \cite{lands}) equals 
\begin{equation}
\label{eq:dis3}
D=\frac{E}{\ln(3+2j)}.
\end{equation}
Figures \ref{fige}, \ref{figd} show the curves 
\begin{equation}
\label{eq:cure}
y(x)=\frac{1}{x+11}\Bigl (10\ln (0.1 x+1.1)+(x+1)\ln (2x+22)\Bigr )
\end{equation}
\begin{equation}
\label{eq:curd}
y(x)=\frac{1}{(x+11)\ln (2x+3)}\Bigl (10\ln (0.1 x+1.1)+(x+1)\ln (2x+22)\Bigr )
\end{equation}
obtained from (\ref{eq:entr3}) and (\ref{eq:dis3}) with specific values: $c_1^2=10$, $c_2^2=c_3^2=0.5$. 
\begin{figure}[htb]
\begin{minipage}{0.42\textwidth}
\centering
\includegraphics[width=2.2in]{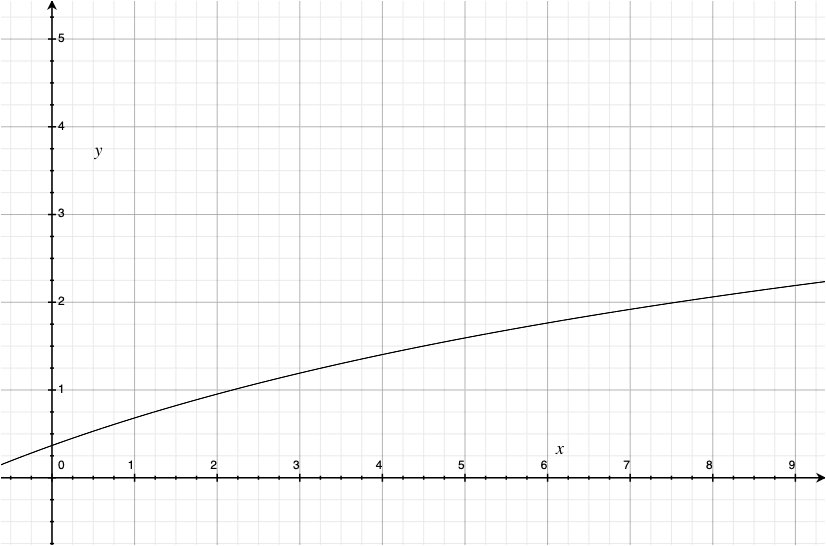}
\caption{The curve (\ref{eq:cure}).} \label{fige}
\end{minipage}
\begin{minipage}{0.42\textwidth}
\centering
\includegraphics[width=2.2in]{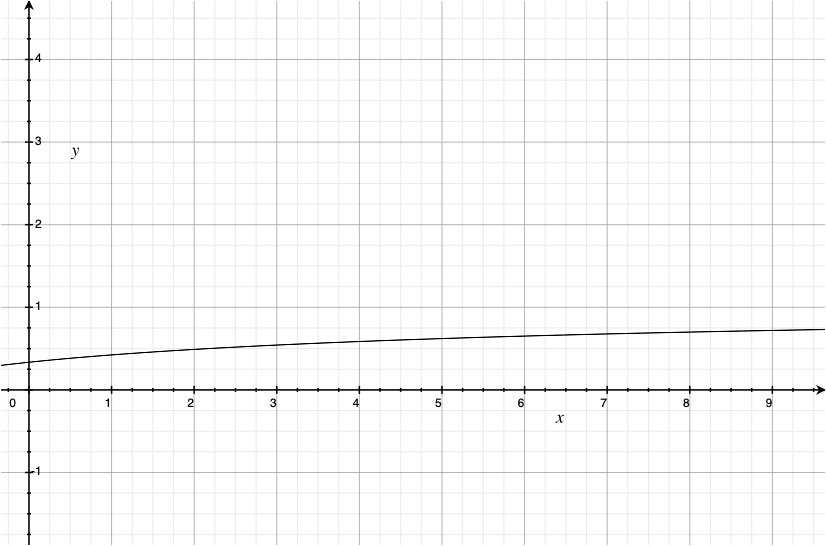}
\caption{The curve (\ref{eq:curd}).} \label{figd}
\end{minipage}
\end{figure}

The table below contains approximate (calculated) values of entropy (\ref{eq:cure}) and disorder (\ref{eq:curd}) for $x\in \{ 1,2,3,4,5\}$.    

\begin{tabular}{|r|r|r|}
  \hline
  &&\\
$j$&$E=y(j)$ from eq.  (\ref{eq:cure}) & $D=\dfrac{E}{\ln(2j+3)}$ \\
&&\\
\hline
1 \  &0.6816&0.4235\\
  \hline
2 \ &0.9537&0.4901  \\
  \hline
  3 \ &1.1924&0.5427   \\
\hline
4 \ &1.4040& 0.5855\\
  \hline
5 \ &1.5934&0.6212\\
\hline
\end{tabular}

\end{example}


\begin{thebibliography}{999}

\bibitem{awane}A. Awane. 
k-symplectic structures. 
J. Math. Phys. 33 (1992), no. 12, 4046-4052.

\bibitem{baezhr}J. Baez, A.  Hoffnung, C.  Rogers. 
Categorified symplectic geometry and the classical string. 
Comm. Math. Phys. 293 (2010), no. 3, 701-725.

\bibitem{barron}T. Barron. Toeplitz operators on K\"ahler manifolds. Examples. Springer Briefs Math. Springer, 2018.


\bibitem{barser}T. Barron, B. Serajelahi. Berezin-Toeplitz quantization, hyperk\"ahler manifolds, and multisymplectic manifolds. Glasgow Math. J. 59, issue 1 (2017), 167-187.

\bibitem{barsha}T. Barron, M. Shafiee. Multisymplectic structures induced by symplectic structures. J. Geom. Phys. 136 (2019), 1-13.

\bibitem{blacker}
C. Blacker. 
Polysymplectic reduction and the moduli space of flat connections. 
J. Phys. A 52 (2019), no. 33, 335201.

\bibitem{blackerq}C. Blacker. 
Quantization of polysymplectic manifolds. 
J. Geom. Phys. 145 (2019), 103480. 

\bibitem{cantr}F. Cantrijn, A.  Ibort,  M. de Le\'on. 
On the geometry of multisymplectic manifolds. 
J. Austral. Math. Soc. Ser. A 66 (1999), no. 3, 303-330.

\bibitem{rupakc}R. Chatterjee. 
Dynamical symmetries and Nambu mechanics. 
Lett. Math. Phys. 36 (1996), no. 2, 117-126.

\bibitem{davison}M. Davison, J. Shiner. 
 Extended entropies and disorder.  Adv. Complex Syst. 8 (2005), no.1, 125-158.

\bibitem{dkarig}K. Dineen, S.  Karigiannis. 
Non-squeezing and capacities for some calibrated geometries. 
Math. Phys. Anal. Geom. 29 (2026), no. 3, Paper No. 31. 

\bibitem{forgomes}M. Forger, L.  Gomes. 
Multisymplectic and polysymplectic structures on fiber bundles. 
Rev. Math. Phys. 25 (2013), no. 9, 1350018. 

\bibitem{gk}V. Ginzburg, M.  Kapranov. 
Koszul duality for operads.
Duke Math. J. 76 (1994), no. 1, 203-272.

\bibitem{rroy}X. Gr\`acia, J.  de Lucas, X.  Rivas, N.  Rom\'an-Roy. 
On Darboux theorems for geometric structures induced by closed forms. 
Rev. R. Acad. Cienc. Exactas Fís. Nat. Ser. A Mat. RACSAM 118 (2024), no. 3, Paper No. 131. 

\bibitem{hirota}Y. Hirota, N.  Ikeda. 
Geometry of bundle-valued multisymplectic structures with Lie algebroids. 
J. Geom. Phys. 203 (2024), Paper No. 105242. 

\bibitem{kimura}T. Kimura, J.  Stasheff, A.  Voronov.  On operad structures of moduli spaces and string theory. Commun. Math. Phys.  171 (1995), 1-25. 

\bibitem{kontsev}M. Kontsevich. Operads and motives in deformation quantization. 
Lett. Math. Phys. 48 (1999), no. 1, 35-72.

\bibitem{lands}P. Landsberg. 
Can entropy and ``order'' increase together ? Phys. Lett. A 102 (1984), no. 4, 171--173.

\bibitem{lein}T. Leinster.  Entropy and Diversity: The Axiomatic Approach. Cambridge University Press; 2021. 

\bibitem{loday}J. Loday, B. Vallette. Algebraic operads. Springer Berlin Heidelberg, Grundlehren der mathematischen Wissenschaften, vol. 346,  2012. 

\bibitem{mathi}O. Mathieu. 
The symplectic operad. In {\it Functional analysis on the eve of the 21st century}, Vol. 1 (New Brunswick, NJ, 1993), pp. 223-243.
Progr. Math. 131, 
Birkh\"auser Boston, Inc., Boston, MA, 1995.

\bibitem{may}J. May. 
The geometry of iterated loop spaces.
Lecture Notes in Math., Vol. 271, 
Springer-Verlag, Berlin-New York, 1972. 

\bibitem{mspera}A. Miti, M.  Spera. 
A hydrodynamical homotopy co-momentum map and a multisymplectic interpretation of higher-order linking numbers. 
J. Aust. Math. Soc. 112 (2022), no. 3, 335-354.

\bibitem{mads1}T. Madsen, A.  Swann. 
Multi-moment maps. 
Adv. Math. 229 (2012), no. 4, 2287-2309.

\bibitem{mads2}T. Madsen, A.  Swann. 
Closed forms and multi-moment maps. 
Geom. Dedicata 165 (2013), 25-52.

\bibitem{nambu}Y. Nambu. 
Generalized Hamiltonian dynamics.
Phys. Rev. D (3) 7 (1973), 2405-2412.

\bibitem{rogers}C. Rogers. 
2-plectic geometry, Courant algebroids, and categorified prequantization. 
J. Symplectic Geom. 11 (2013), no. 1, 53-91.

\bibitem{rw}L. Ryvkin, T.  Wurzbacher. 
An invitation to multisymplectic geometry. 
J. Geom. Phys. 142 (2019), 9-36.

\bibitem{ss}C. S\"amann, R.  Szabo. 
Groupoids, loop spaces and quantization of 2-plectic manifolds. 
Rev. Math. Phys. 25 (2013), no. 3, 1330005. 

\end{thebibliography}
\end{document}